\documentclass{amsart}
\usepackage{comment}
\usepackage{hyperref}
\usepackage{amssymb}
\usepackage{mathrsfs}%
\usepackage[title]{appendix}%
\usepackage{xcolor}%
\usepackage{textcomp}%
\usepackage{manyfoot}%
\usepackage{booktabs}%
\usepackage{algorithm}%
\usepackage{algorithmicx}%
\usepackage{algpseudocode}%
\usepackage{listings}%

\usepackage{tikz-cd}
\usepackage{tikz-3dplot}
\usepackage{pgfplots}
\pgfplotsset{compat=1.17}
\usepackage{pgfplotstable}
\usepgfplotslibrary{patchplots}
\usetikzlibrary{3d,calc}
\usetikzlibrary{calc, decorations.pathreplacing}
\usepackage{float}
\usepackage[T1]{fontenc}
\usepackage[utf8]{inputenc}
\usepackage[english]{babel}

\newcommand{\numberset}{\mathbb}
\newcommand{\N}{\numberset{N}}
\newcommand{\R}{\numberset{R}}
\newcommand\ddfrac[2]{\frac{\displaystyle #1}{\displaystyle #2}}

\theoremstyle{plain} 
\newtheorem{thm}{Theorem}[section]
 
\newtheorem{lem}[thm]{Lemma} 
\newtheorem{prop}[thm]{Proposition} 
\newtheorem*{theorem*}{Theorem}
\newtheorem{op}{Problem}

\theoremstyle{definition}

\theoremstyle{remark}

\title[A self-contained proof of the ACF formula]{A self-contained proof of the Alt-Caffarelli-Friedman monotonicity formula}

\author{Emanuele Salato}
\address{Dipartimento di Scienze Matematiche, Politecnico di Torino, Corso Duca degli Abruzzi 24, 10129 Torino, Italy
}
\address{Dipartimento di Matematica ``G. Peano'', Università di Torino, Via Carlo Alberto 10, 10123 Torino, Italy.}
\address{Laboratoire de Mathématiques, Université Savoie Mont Blanc, UMR CNRS 5127, Campus Scientifique, 73376 Le Bourget-du-Lac, France}
\email{emanuele.salato@unito.it}

\begin{document}
\begin{abstract}
    The Alt-Caffarelli-Friedman monotonicity formula is a cornerstone in the theory of free boundary problems. In this note we provide a self-contained proof of this result. To prove the main stepping stone, namely the \emph{Friedland-Hayman inequality}, we exploit a useful convexity property.
\end{abstract}

\maketitle
\tableofcontents

{\small


}

\section{Introduction}
The \emph{Alt-Caffarelli-Friedman monotonicity formula} (ACF formula) was first introduced in \cite[Lemma 5.1]{acf} as a tool designed to deal with the regularity of solutions of a particular two-phase free boundary problem.
A crucial issue in facing these kind of problems is to establish the optimal regularity of solutions across the free boundary. In this setting, the ACF formula is a valuable device, together with regularity techniques, to provide estimates for the behavior of the gradient of a solution of the problem in a point of the free boundary, taking into account of the contribution of the two phases.

Given the dimension $n$, we can depict the typical situation in the ball with radius two $B_2 \subset \R^n$, with center in the origin $\textbf{0}$ (in the following the subscript of $B$ will refer to its radius). The non negative functions $u_+,u_- \in C(B_2)$ appearing in the ACF formula need to satisfy
\begin{equation}\label{ACFcond}
    \begin{cases}
    \Delta u_+ \ge 0 & \,  \textrm{in} \,\, \{u_+>0\}, \\
    \Delta u_- \ge 0 & \,  \textrm{in} \,\, \{u_->0\}, \\
    
    u_+(x) \cdot u_-(x)=0 & \,  \text{for every } x \in B_2,\\
    u_+(\textbf{0})=u_-(\textbf{0})=0,
    \end{cases}
\end{equation}
where the first two inequalities hold in the sense of distributions (see Figure \ref{f.ACFconfig}).

\begin{figure}[htp]
\centering
\begin{tikzpicture}
  % Outer and inner circles
  \draw[thick] (0,0) circle (3cm);
  \draw[dashed] (0,0) circle (1.5cm);

  % Bezier curve
  \coordinate (A) at (-3,0);
  \coordinate (B) at (3,0);
  \coordinate (C1) at (-1,2);
  \coordinate (C2) at (1,-2);
  \draw[thick] (A) .. controls (C1) and (C2) .. (B);

  % All labels as \node at (...) {...}
  \node at (-1.5,1.6) {$u_+ > 0$};        % near upper part of curve
  \node at (-1.5,-1.6) {$u_- > 0$};       % near lower part of curve
  \node at (2.5,-2.5) {$B_2$};
  \node at (1.9,0) {$B_1$};
  \node at (1.4,1.7) {$\Delta u_+  \ge  0$};
  \node at (1.4,-1.7) {$\Delta u_-  \ge  0$};

  % Origin
  \fill (0,0) circle (2pt); 
  \node[above right] at (0,0) {\textbf{0}};
\end{tikzpicture}
\caption{A pair $u_+,u_-$ satisfying \eqref{ACFcond} in $\R^2$.}
\label{f.ACFconfig}
\end{figure}
The theory developed in \cite{acf} has been designed to understand the model studied in \cite{acf1}, which involves the irrotational flow of two ideal (incompressible) fluids. In terms of their (suitably normalized) stream functions, these properties translate into a null divergence condition.

Despite being a frequently used and very well-known object in the field of free boundary problems, in the wide literature concerning this subject, we were not able to find a self-contained comprehensive proof.
This probably happened because the central fact necessary to obtain monotonicity, i.e., the \emph{Friedland-Hayman inequality}, obtained as a corollary of \cite[Theorem 3]{fh}, had not at that time been demonstrated in a totally analytical way. This result is a sharp inequality concerning the growth rates of two homogeneous harmonic functions with Dirichlet boundary conditions on disjoint cones of the Euclidean space. It plays a central role in the proof of the ACF formula, since it is possible to associate $u_+$ and $u_-$ with two such functions and to obtain through the inequality a lower bound on the growth rate of the ACF formula.
The original proof of this powerful tool relies on a result achieved using numerical techniques, see \cite{ho}, which involves \emph{Hermite's functions}. A way to analytically complete this proof of the Friedland-Hayman inequality is provided in \cite[Section 8]{em}, which, however, contains rather involved calculations. 

A more detailed version of the proof of the ACF formula was given, following a different approach that exploits the one-dimensional \emph{Gaussian measure}, in \cite[Chapter 12]{cs}. Indeed, an easier proof of the two-dimensional case is provided, while in dimension greater than or equal to three a more accurate analysis and refined tools are needed. However, a part of this proof relies on the unpublished paper \cite{bkp} (the contents of this work are sketched in \cite[Section 2.4]{ck}).

The purpose of this paper is to give a self-contained and comprehensive proof of the ACF formula in the case $n \ge 3$, with a different approach to the Friedland-Hayman inequality.
The proof we present is \emph{not original}, its structure is the same as \cite{n}. We exploit the content of \cite{k} to obviate a flaw present in the proof of a convexity property. In particular instead of \cite[Teorema 4.5]{n} we use Proposition \ref{propconvex}, which corrects this result to the extent necessary for our purposes. We have deliberately decided to avoid technicalities related to some regularity issues in order to make the presentation more immediate, while providing adequate references when necessary.

\begin{thm}[ACF formula]\label{t.ACF}
Let $u_+,u_-$ be as in \eqref{ACFcond}. The function $J: (0,1) \rightarrow \R$, defined by
\begin{equation} \label{acffunctional}
    J(s)=\frac{1}{s^4} \int_{B_s}\frac{|\nabla u_+(x) |^2}{|x|^{n-2}} \, dx \int_{B_s}\frac{|\nabla u_-(x) |^2}{|x|^{n-2}} \, dx
\end{equation}
for every $s \in (0,1)$, is finite and increasing. 
\end{thm}
This result can be proved requiring weaker assumptions on the configuration described by \eqref{ACFcond}, see for example \cite[Theorem 1.3]{v}, where the functions are only requested to belong to the Sobolev space $H^1(B_1)$, instead of being continuous.

We recall that in \cite[Theorem 2.9]{psu} the rigidity cases, namely the ones were the function appearing in the formula is constant, are analyzed (this problem was originally considered in \cite[Section 6]{acf}).
\begin{thm}[Rigidity of the ACF formula]\label{t.rigidity}
    Let $u_+,u_-$ be as in \eqref{ACFcond} and $J$ be as in \eqref{acffunctional}. Assume that $J(R_1)=J(R_2)$ for some $0<R_1<R_2<1$, then either of the following holds:
    \begin{itemize}
        \item if $J(R_2)=0$, then $u_+\equiv0$ in $B_{R_2}$ or $u_-\equiv0$ in $B_{R_2}$;
        \item if $J(R_2)>0$, then there exist a unit vector $\nu \in \partial B_1$ and two constants $c_+,c_->0$ such that
        \begin{equation*}
            u_+(x)=c_+(x \cdot \nu)^+ \quad , \quad
            u_-(x)=c_-(x \cdot \nu)^- \quad \textup{for } x \in B_{R_2},
        \end{equation*}
        where $(x \cdot \nu)^+$ and $(x \cdot \nu)^-$ are the positive part and the negative part of the function $x \mapsto x \cdot \nu$, defined for every $x \in B_1$.
    \end{itemize}
\end{thm}
We notice that a stability result related to the ACF formula was given in \cite[Theorem 1.3]{akn}.
\begin{thm}[Stability of the ACF formula]\label{t.stability}
Let $\rho \in [0,1/2]$. There exists a constant $C=C(n)>0$ such that the following holds. Suppose that $u_+,u_-$ satisfy \eqref{ACFcond} and let $J$ be as in \eqref{acffunctional}. Then there exist two constants $c_+,c_->0$ and a unit vector $\nu \in \partial B_1$ such that
\begin{equation*}
    \int_{B_1 \setminus B_\rho} \!\!\!\!\!\!\!\! \big[(u_+(x) - c_+(x \cdot \nu)^+)^2 +  (u_-(x) - c_-(x \cdot \nu)^-)^2 \big]\, dx
    \le C \log \! \left( \frac{J(1)}{J(\rho)} \right) \!\! \left( ||u_+||^2_{2} + ||u_-||^2_{2} \right),
\end{equation*}
where we denote with $||\cdot||_2$ the $L^2$-norm of a function in $B_1$. Furthermore, there exists a dimensional constant $\epsilon_0= \epsilon_0(n)\!> \!0$ such that if the quotient $\log(J(1)/J(0^+))\!<\!\epsilon_0$, then $c_+,c_-\!$ and $\nu$ may be chosen independently from $\rho$.
\end{thm}
We are aware of the fact that in the literature many other generalizations of the ACF formula are available. We have chosen to present only these two results for the sake of shortness, because they are prime examples of the so called \emph{rigidity} and \emph{stability} statements.

The paper is organized as follows. In Section \ref{s.2.1} we recall some useful fact regarding changes of coordinates and rearrangements. In Section \ref{s.2.2} we study the first Dirichlet eigenvalue on open subsets of the sphere. In Section \ref{s.3} we prove a convexity property. In Section \ref{s.4} we give the proof of Theorem \ref{t.ACF}. In Section \ref{s.5} we present an open problem.

\section{Parametrizations and rearrangements}\label{s.2.1}
Consider the manifold $(\R^n,g_{\R^n})$, where $g_{\R^n}$ denotes the (standard) \emph{flat} Euclidean metric. We have that the $(n-1)$-dimensional sphere of unit radius $\partial B_1$ can be endowed with the Riemannian metric inherited from $(\R^n,g_{\R^n})$, that we call \emph{round} and denote $g_{\partial B_1}$. So it is possible to define the notions of \emph{Riemannian gradient} $\nabla_\phi$ and \emph{Laplace-Beltrami operator} $\Delta_\phi$ on $(\partial B_1,g_{\partial B_1})$, see \cite[Definitions 1 and 3, p. 2-3]{c}.

The \emph{polar parametrization} of $\R^n$, with respect to the origin $\textbf{0}$, is given by the function
\begin{equation}\label{polarparam}
    \begin{split}
        \mathscr{P}:\R^+ \times \partial B_1 &\rightarrow \R^n \\
        (r,\phi) &\mapsto x=r\phi,
    \end{split}
\end{equation}
we call the parameters in $\R^+ \times \partial B_1$ \emph{polar coordinates} and the first one \emph{radial coordinate}. 
It is know that $(\R^n\setminus \{\textbf{0} \},g_{\R^n})$ is isometric to $(\R^+ \times \partial B_1, g_{\R^+} + r^2g_{\partial B_1})$, see \cite[Section 1.4.4]{pet}.
Let $F$ be a function with domain $\R^n$, we define $F_\mathscr{P}:=F \circ \mathscr{P}$.
Let $v \in L^1(\R^n)$ by the changes of variables formula we have 
\begin{equation}\label{eq:polarchange}
    \int_{\mathbb{R}^n}  v(x) \, dx = \!\int_0^{+\infty}  \int_{\partial B_1} v_\mathscr{P}(r,\phi) \, r^{n-1} \, d\sigma_\phi \, dr,
\end{equation}
where $d\sigma_\phi$ is $(n-1)$-dimensional Hausdorff measure on  $\partial B_1$, see \cite[Theorem 2.49]{fol}. From now on we will write $d\sigma$ instead of $d\sigma_\phi$. Let $\textbf{r}$ be the vector of the orthonormal frame of $\R^+ \times \partial B_1$ corresponding to the radial coordinate. Let us also assume that
$v$ is a twice differentiable function. We can express its gradient in polar coordinates via an orthogonal decomposition (that relies on the orthonormal frame) as
\begin{equation}\label{Gradient in polar coordinates}
    (\nabla v)_\mathscr{P}=(v_\mathscr{P})_r \, \textbf{r} + \frac{1}{r}\nabla_\phi (v_\mathscr{P}),
\end{equation}
where the subscript $r$ denotes the differentiation with respect to the radial coordinate, see \cite[equation (1.4.6)]{fy}. Similarly its Laplacian can be written as 
\begin{equation}\label{Laplacian in polar coordinates}
    (\Delta v)_\mathscr{P}=(v_\mathscr{P})_{rr}+\frac{n-1}{r}(v_\mathscr{P})_r+\frac{1}{r^2}\Delta_\phi (v_\mathscr{P}),
\end{equation}
see \cite[Lemma 1.4.1]{fy}. 

The \emph{hypershperical parametrization} of $\partial B_1$, with respect to the \emph{north pole} $p:=(1, 0,\ldots,0)$, is given by the function 
\begin{equation}\label{sphereparam}
    \begin{split}
           \mathscr{S}: (0,\pi) \times \mathbb{S}^{n-2}&\rightarrow \partial B_1 \\
        (\theta,\xi) &\mapsto \phi=(\cos(\theta) ,\sin(\theta) \xi),
    \end{split}
\end{equation}
where we define $\mathbb{S}^{n-2}:=\partial B_1 \cap \{x_1=0\}$. 
\begin{figure}[H]
    \centering
    \hspace{0.5cm}
\begin{tikzpicture}[scale=1]

% Parameters
\def\Ra{2}     % Radius of sphere

% Center the drawing by shifting left by Ra
\begin{scope}[shift={(-100*\Ra,0)}]

  % Draw full circle (sphere cross-section) in black
  \draw[thick, black] (0,0) circle(\Ra);

  % Draw equator ellipse (horizontally across the center) in blue
  \draw[thick, blue] (0,0) ellipse({\Ra} and {0.3});

  % Add label for the origin
  \fill (0,0) circle (1pt);
  \node[right] at (0.1,0) {$\mathbf{0}$};

  % Add label for the north pole p
  \fill (0,\Ra) circle (1pt);
  \node[above right] at (0,\Ra) {$p$};

  % Add label near the right edge of the sphere and ellipse
  \node[right] at (\Ra + 0.2, 0) {$\mathbb{S}^{1}$};

  % Optional: label at bottom for full boundary
  \node at (2,-\Ra) {$\partial B_1$};

\end{scope}

\end{tikzpicture}

    \caption{The circumference of unit radius $\mathbb{S}^{1}$, in $\R^3$ (in blue).}
    \label{}
\end{figure}
We call the parameters in $(0,\pi) \times \mathbb{S}^{n-2}$ \emph{hypershperical coordinates} and the first one \emph{colatitude coordinate}. Notice that it is possible to obtain an explicit expression for the colatitude coordinate, that from now on we will call colatitude, namely
$$\theta=\arccos(p \cdot \phi).$$
This quantity, geometrically, represents the amplitude (in radiant) of the angle between the two radii of $\partial B_1$ connecting the origin $\textbf{0}$ with the points $p$ and $\phi$, respectively.
\begin{figure}[H]
    \centering
\begin{tikzpicture}[scale=1]
% Parameters
\def\Ra{2}         % Radius of sphere
\def\tcol{30}     % Numerical value for internal computation
\pgfmathsetmacro{\x}{0.7*\Ra*sin(\tcol)}  % Point q inside the sphere (scaled radius)
\pgfmathsetmacro{\y}{0.7*\Ra*cos(\tcol)}

% Draw full circle (sphere cross-section) in black
\draw[thick, black] (0,0) circle(\Ra);

% Draw equator ellipse (dashed, black)
\draw[dashed, gray] (0,0) ellipse({\Ra} and {0.3});

% Add dashed line from origin to north pole
\draw[dashed] (0,0) -- (0,\Ra);

% Add label for the origin
\fill (0,0) circle (1pt);
\node[right] at (0.1,0) {$\mathbf{0}$};

% Add label for the north pole p
\fill (0,\Ra) circle (1pt);
\node[above right] at (0,\Ra) {$p$};

% Add a point q inside the sphere (not on boundary)
\fill (\x,\y) circle (1pt);
\node[above right] at (\x,\y) {$\phi$};

% Draw dashed radius line from origin to q
\draw[dashed] (0,0) -- (\x,\y);

% Draw colatitude angle arc labeled theta_0
\draw[->] (0,0) ++(90:0.4) arc[start angle=90, end angle={90 - \tcol}, radius=0.4];
\node at ({0.6*cos(90 - \tcol/2) +0.05}, {0.6*sin(90 - \tcol/2)+0.2}) {$\theta_0$};

% Optional: label at bottom for full boundary
\node at (2,-\Ra) {$\partial B_1$};

\end{tikzpicture}

    \caption{A point $\phi$ of colatitude $\theta_0 \approx \pi/6$, in $\partial B_1 \subset \R^3$.}
    \label{}
\end{figure}  
Let $\theta_0 \in (0,\pi)$, we define the open set
\begin{equation*}
        \Gamma(\theta_0):=\{\phi \in \partial B_1 \, :  \, 0 \le \arccos(p \cdot \phi) < \theta_0 \}.
    \end{equation*}
A \emph{spherical cap} of colatitude $\theta_0$ (with center $\psi(p)$, in $\partial B_1 \subset \R^n$) is a set of the type $\psi(\Gamma(\theta_0))$, where $\psi: \partial B_1 \rightarrow \partial B_1 $ is an isometry of $\partial B_1$. Recall that the isometry group of $\partial B_1$ is given by \emph{the orthogonal group} $O(n)$, that is made up of composition of the so-called \emph{rotations} and \emph{reflections}, see \cite[Problem 5-8, p. 88]{lee}.
    
\begin{figure}[htp]
    \centering
\begin{tikzpicture}[scale=1]

% Parameters
\def\Ra{2}     % Radius of sphere
\def\h{0.6}   % Height of cap
\pgfmathsetmacro{\a}{sqrt(\Ra*\Ra - (\Ra - \h)*(\Ra - \h))}  % Cap base radius
\pgfmathsetmacro{\angle}{acos((\Ra - \h)/\Ra)}             % Colatitude in degrees

% Draw full circle (sphere cross-section)
\draw[thick] (0,0) circle(\Ra);

% Draw base ellipse of cap
\draw[thick] (0,\Ra - \h) ellipse({\a} and {0.1});

% Vertical line from center to cap center
\draw[dashed] (0,0) -- (0,\Ra - \h);

% Radius line from center to edge of cap base
\draw[dashed] (0,0) -- ({\a},{\Ra - \h});

% Draw the angle arc (from vertical to cap edge direction)
\draw[->] (0,0) ++(90:0.4) arc[start angle=90, end angle={90 - \angle}, radius=0.4];
\node at ({0.6*cos(90 - \angle/2)}, {0.6*sin(90 - \angle/2)}) {$\theta_0$};

% Add the label Gamma(theta_0) near the elliptical base
\node[right] at ({\a + 0.1},{\Ra - \h}) {$\Gamma(\theta_0)$};

\node at (2,-\Ra ) {$\partial B_1$};

% Add label for the origin
\fill (0,0) circle (1pt);
\node[below left] at (0,0) {$\mathbf{0}$};

% Add label for the north pole p
\fill (0,\Ra) circle (1pt);
\node[above right] at (0,\Ra) {$p$};
\end{tikzpicture}
    \caption{The set $\Gamma(\theta_0)$ with $\theta_0 \approx \pi/4$ (in $\partial B_1 \subset \R^3$).}
    \label{}
\end{figure}
It is know that $(\partial B_1,g_{\partial B_1})$ is isometric to $((0,\pi) \times \mathbb{S}^{n-2}, g_{(0,\pi)} + (\sin\theta)^2g_{\mathbb{S}^{n-2}})$, where $g_{\mathbb{S}^{n-2}}$ is the round metric on $\mathbb{S}^{n-2}$, see \cite[Example 1.4.6]{pet}. It is possible to define the notions of Riemannian gradient $\nabla_\xi$ and Laplace-Beltrami operator $\Delta_\xi$ on $(\mathbb{S}^{n-2},g_{\mathbb{S}^{n-2}})$. Let $f$ be a function with domain $\partial B_1$, we define $f_\mathscr{S}=f \circ \mathscr{S}$.
Let $u \in L^1(\partial B_1)$  by the changes of variables formula we have that it holds
\begin{equation}\label{eq:sphericalchange}
    \int_{\partial B_1}  u(\phi) \, d\sigma  =  \int_0^\pi \int_{\mathbb{S}^{n-2}}  u_\mathscr{S}(\theta,\xi) \sin^{n-2}(\theta) \, d\sigma_\xi \, d\theta,
\end{equation}
where $d\sigma_\xi$ is $(n-2)$-dimensional Hausdorff measure on  $\mathbb{S}^{n-2}$, see \cite[equation (1.5.4)]{fy}. Let $\boldsymbol{\theta}$ be the vector of the orthonormal frame of $(0,\pi) \times \mathbb{S}^{n-2}$ corresponding to the colatitude coordinate. Let us also assume that
$u$ is a twice differentiable function. We can express its gradient in hypershperical coordinates via an orthogonal decomposition (that relies on the orthonormal frame) as
\begin{equation}\label{Gradient in spherical coordinates}
    (\nabla u)_\mathscr{S}=(u_\mathscr{S})_\theta \, \boldsymbol{\theta} + \frac{1}{\sin (\theta)}\nabla_\xi (u_\mathscr{S}),
\end{equation}
where the subscript $\theta$ denotes the differentiation with respect to the colatitude. Similarly its Laplace-Beltrami operator can be written as
\begin{equation}\label{Laplacespheretheta}
    (\Delta_\phi u)_\mathscr{S}=\frac{1}{(\sin \theta)^{n-2}}  \bigg( (\sin \theta)^{n-2} (u_\mathscr{S})_{\theta} \bigg)_{\theta} + \frac{1}{(\sin \theta)^2} \Delta_\xi (u_\mathscr{S}),
\end{equation}
see \cite[Lemma 1.4.2]{fy}.

Combining \eqref{polarparam} and \eqref{sphereparam} it is possible to find another parametrization of $\R^n$, given by the function
\begin{equation}\label{polarsphereparam}
    \begin{split}
        \R^+ \times (0,\pi) \times \mathbb{S}^{n-2} &\rightarrow \R^n \\
        (r,\theta,\xi) &\mapsto \mathscr{P}(r,\mathscr{S}(\theta,\xi))=(r\cos(\theta) ,r\sin(\theta) \xi).
    \end{split}
\end{equation}
The \emph{stereographic projection} of $\partial B_1$, with respect to the \emph{south pole} $-p$, is given by the function 
\begin{equation*}
    \begin{split}
           \mathscr{X}: \partial B_1 \setminus \{-p\} &\rightarrow \R^{n-1} \\
        \phi=(\hat\phi,\phi_n) &\mapsto \hat x = \frac{\hat\phi}{1 + \phi_{n}}.
    \end{split}
\end{equation*}
It is know that $(\partial B_1 \setminus \{-p\}, g_{\partial B_1})$ is isometric to $(\R^{n-1}, \Upsilon^2 g_{\R^{n-1}})$, where the analytic function $\Upsilon: \R^{n-1} \rightarrow \R$ is defined by
$$\Upsilon(\hat x) = \frac{2}{1 + |\hat x|^2}$$
for every $\hat x \in \R^{n-1}$, see \cite[Proof of Lemma 3.4, p. 37]{lee}.
Let $v: \R^{n-1} \rightarrow \R$ be a function and define $v_{\mathscr{X}^{-1}}:=v \circ\mathscr{X}^{-1}$.
Suppose in addition that $v$ is twice differentiable and denote $\Delta_\Upsilon$ the Laplace-Beltrami operator induced on $\R^{n-1}$ by the metric $\Upsilon^2 g_{\R^{n-1}}$, it holds
\begin{equation}\label{Laplacestereo}
    \Delta_\Upsilon v=\frac{1}{\Upsilon^2} \left( \Delta v + (n-2) \nabla v  \cdot \nabla \log \Upsilon \right). 
\end{equation}
It is possible to retrieve this relation by direct calculation exploiting \cite[equation (33), p. 5]{c}.

In the following it will be useful to produce functions on $\partial B_1$ that are radially symmetric (with respect to a point of $\partial B_1$), namely whose superlevel sets are spherical caps with the same center, see \cite[Sections 1.1, 1.2 and 7.1]{bdl}. We notice that each of these functions can be obtained by as a composition of a map that is radially symmetric with respect to $p$ and an isometry of $\partial B_1$. We denote $|\cdot|$, with a little abuse with respect to the norm of a vector, the $(n-1)$-dimensional Hausdorff measure on $\partial B_1$.
We recall that, given a non-negative measurable function $u : \partial B_1 \rightarrow \R$, its \emph{distribution function}, i.e., the map that describes the measure of the superlevel sets, $\mu_u: [0,\infty) \rightarrow[0,|\partial B_1|]$ is defined as
    \begin{equation*}
        \mu_u(t)=|\{\phi \in \partial B_1 \, : \, u(\phi)> t\}|
    \end{equation*}
    for every $0 \le t < +\infty$. In addition we can define its 
    \textit{decreasing rearrangement} $u^*:[0,|\partial B_1|] \rightarrow [0,\infty)$ as
    \begin{equation*}
         u^*(s)=\inf\{ t \ge 0 \, : \, \mu_u(t)\le s\}
    \end{equation*}
    for every $s \in [0,|\partial B_1|]$. Now it is useful to introduce the map $M: \partial B_1 \rightarrow [0,|\partial B_1|]$, defined as 
    $$M(\phi)=|\Gamma(\arccos( p \cdot \phi))|=|\Gamma(\theta)|$$
    for every $\phi \in \partial B_1$. Finally, the \textit{symmetric decreasing rearrangement} $u^\#: \partial B_1 \rightarrow (0,+\infty)$ is defined as
    $$u^\#(\phi)=( u^* \circ M)(\phi)$$
    for every $\phi \in \partial B_1$ (extended by $0$ where $u^*$ is not defined). The function $u^\#$ is radially symmetric with respect to $p$ and decreases as the colatitude increases.

\begin{figure}[htp]
    \centering
   \begin{tikzpicture}[scale=1]

% Parameters
\def\Ra{2}     % Radius of sphere
\def\h{0.6}    % Height of cap
\pgfmathsetmacro{\a}{sqrt(\Ra*\Ra - (\Ra - \h)*(\Ra - \h))}  % Radius of cap base

% Draw full circle (sphere cross-section)
\draw[thick] (0,0) circle(\Ra);

% Fill the cap area, including the elliptical base
\fill[blue!30, opacity=0.5]
  (-\a,\Ra - \h)
  arc[start angle=180, end angle=0, radius=\a]
  -- (\a,\Ra - \h)
  arc[start angle=0, end angle=180, x radius=\a, y radius=0.1]
  -- cycle;

% Draw horizontal ellipse (base of the cap)
\draw[thick] (0,\Ra - \h) ellipse({\a} and {0.1});

% Add label for the boundary
\node at (2,-\Ra ) {$\partial B_1$};

% Add label for the origin
\fill (0,0) circle (1pt);
\node[below left] at (0,0) {$\mathbf{0}$};

% Add label for the north pole p
\fill (0,\Ra) circle (1pt);
\node[above right] at (0,\Ra) {$p$};

\end{tikzpicture}
\vspace{5mm}
    \caption{The graph of a function that is radially symmetric with respect to $p$ (in violet) decreasing as the colatitude increases.}
    \label{}
\end{figure}

\section{The first Dirichlet eigenvalue}\label{s.2.2}
For the first part of this section we follow \cite{c}, where a regularity hypothesis on the boundary of a set and connectedness of the set are also required, see \cite[p. 1]{c}. Since the results and the definitions we borrow are valid even without these assumptions (with the appropriate modifications), we continue to refer to \cite{c}.

Let $\Gamma \subset \partial B_1$ be an open set, the \emph{first Dirichlet eigenvalue of $\Gamma$}, see \cite[p. 17]{c}, is
\begin{equation}\label{spherepseudoeigen}
    \lambda(\Gamma)=\min_{\substack{u \in H_0^1(\Gamma) \\ u \not \equiv 0}} \ddfrac{\int_{\Gamma} |\nabla_\phi u(\phi)|^2 \, d\sigma}{\int_\Gamma u(\phi)^2 \, d\sigma}.
\end{equation}
A function $u$ achieving equality in \eqref{spherepseudoeigen} is called \emph{(first Dirichlet) eigenfunction corresponding to $\lambda(\Gamma)$} and by regularity theory (combining \eqref{Laplacestereo} and \cite[Analyticity Theorem, p. 136]{bjs}), we have that $u$ is analytic in $\Gamma$ and solves
\begin{equation}\label{eq:eigen}
    -\Delta_\phi u(\phi) = \lambda(\Gamma) u(\phi) \quad \text{for every } \phi \in\Gamma.
\end{equation}
The eigenfunctions are, as usual, extended by zero in $\partial B_1 \setminus \Gamma$. By the \emph{Courant's nodal domain Theorem}, see \cite[p. 19]{c}, an eigenfunction corresponding to $\lambda(\Gamma)$ has only a nodal domain, namely the set $\Gamma \setminus\{u=0\}$ has only a connected component. In particular, when $\Gamma$ is connected, the eigenfunction corresponding to $\lambda(\Gamma)$ is unique, up to scalar multiples (in this case the eigenvalue is called \emph{simple}). In addition, let $\psi: \partial B_1 \rightarrow \partial B_1 $ be an isometry (of $\partial B_1$), we notice that $\lambda(\Gamma)=\lambda(\psi(\Gamma))$.

A central role in the proof of the ACF formula will be played by the properties of the function describing the first Dirichlet eigenvalue of a spherical cap in terms of its colatitude, so we define
    $\lambda(\theta_0):=\lambda(\Gamma(\theta_0))$.
In the case $\theta_0=\pi/2$, by \eqref{polarsphereparam}, we notice that the function $u : \Gamma(\pi/2) \rightarrow \R$, defined by 
$$\quad u_\mathscr{S}(\theta,\xi)=\cos(\theta),$$
for every $(\theta,\xi) \in \mathscr{S}^{-1}(\Gamma(\pi/2))$, is the restriction to the sphere $\partial B_1$ of the positive part of the first euclidean coordinate. In particular $u$ is positive in $\Gamma(\pi/2)$ and vanishes on $\partial \Gamma(\pi/2)$. By \eqref{Laplacespheretheta} we have that $u$ solves \eqref{eq:eigen} with 
\begin{equation}\label{eigenhalfcap}
    \lambda(\pi/2)=n-1,
\end{equation}
so it is an eigenfunction corresponding to $\lambda(\pi/2)$.

We present two useful facts regarding homogeneous functions.
To do so we define the \emph{cone generated by $\Gamma$} with vertex in the origin (of $\R^n$) as the set
$$\{x \in \R^n \, : x =r\phi ,\, \textup{with } \phi \in \Gamma \textup{ and } r \in \R^+ \}.$$

\begin{figure}[H]
    \centering
    
   \begin{tikzpicture}[scale=1]

% Parameters
\def\Ra{2}     % Radius of the unit sphere
\def\h{0.6}    % Height of the cap
\def\Rmax{3}   % Maximum radius to which the cone extends (outside sphere)

% Compute base radius of the cap and angle
\pgfmathsetmacro{\a}{sqrt(\Ra*\Ra - (\Ra - \h)*(\Ra - \h))}
\pgfmathsetmacro{\angle}{acos((\Ra - \h)/\Ra)}

% Draw full circle (sphere cross-section)
\draw[thick] (0,0) circle(\Ra);

% Draw base ellipse of cap (cross-section)
\draw[thick] (0,\Ra - \h) ellipse({\a} and {0.1});

% Label the cap
\node[right] at ({\a + 0.1},{\Ra - \h}) {$\Gamma$};
\node at (2,-\Ra ) {$\partial B_1$};

% Add label for the origin
\fill (0,0) circle (1pt);
\node[below left] at (0,0) {$\mathbf{0}$};

% Compute unit vectors for the boundary rays
\pgfmathsetmacro{\norm}{sqrt(\a*\a + (\Ra - \h)*(\Ra - \h))}
\pgfmathsetmacro{\uxL}{-1*\a/\norm}
\pgfmathsetmacro{\uyL}{(\Ra - \h)/\norm}
\pgfmathsetmacro{\uxR}{\a/\norm}
\pgfmathsetmacro{\uyR}{(\Ra - \h)/\norm}

% Draw only the two boundary lines of the cone (thick and solid)
\draw[thick, orange!80!black] (0,0) -- ({\Rmax*\uxL}, {\Rmax*\uyL});
\draw[thick, orange!80!black] (0,0) -- ({\Rmax*\uxR}, {\Rmax*\uyR});

% Optional: Fill extended cone sector
\path[clip] 
  (0,0) -- ({\Rmax*\uxL}, {\Rmax*\uyL}) 
  arc[start angle=180 - \angle, end angle=\angle, radius=\Rmax] -- cycle;
\fill[orange!20, opacity=0.3] (0,0) circle(\Rmax);

\end{tikzpicture}
    
    \caption{The cone generated by a set $\Gamma \subset \partial B_1$ with vertex in the origin (of $\R^3$).}
    \label{fig:enter-label}
\end{figure}

\begin{prop}\label{p.harmextens}
    Let $\alpha>0$, $u$ be an eigenfunction corresponding to $\lambda(\Gamma)$ and $w: \R^n\rightarrow \R$ be the function defined by
            $$w_{\mathscr{P}}(r,\phi)=r^{\alpha} u(\phi),$$
    for every $(r,\phi) \in \R^+ \times \partial B_1$. There exists a unique positive value $\alpha=\alpha(\Gamma)$, called the \textit{characteristic constant} of the set $\Gamma$, such that $w$ is harmonic in the cone generated by $\Gamma$ with vertex in the origin. This values satisfies
    \begin{equation}\label{def.alpha}
    \alpha(\Gamma)=\sqrt{\bigg( \dfrac{n-2}{2} \bigg)^2 + \lambda(\Gamma)}- \dfrac{n-2}{2}.
\end{equation}
\end{prop}
\begin{proof}
    Requiring $w$ to be harmonic, by \eqref{Laplacian in polar coordinates} we have
\begin{equation*}
    0=r^{\alpha-2}\big[\alpha(\alpha-1)+\alpha(n-1) \big] u(\phi) +r^{\alpha-2}\Delta_\phi u(\phi)
\end{equation*}
for every $(r,\phi) \in \R^+ \times \partial B_1$, that is equivalent by \eqref{eq:eigen} to
\begin{equation*}
    \alpha^2+(n-2)\alpha-\lambda(\Gamma)=0,
\end{equation*}
giving the desired conclusion.
\end{proof}

As a consequence of \eqref{eq:polarchange} and \eqref{Gradient in polar coordinates} we have the following result.
\begin{lem}\label{l.inthom}
Let $\alpha>0$, $u: \partial B_1 \rightarrow \R$ be a differentiable function and $w :\R^n  \rightarrow \R$ be the function defined by
$$w_{\mathscr{P}}(r,\phi)=r^\alpha u(\phi)$$
for every $ (r,\phi)\in \R^+ \times \partial B_1 $. Then
\begin{gather*}
        \int_{B_1} \frac{|\nabla w(x)|^2}{|x|^{n-2}}  \, dx = \frac{1}{2\alpha} \int_{\partial B_1} \left[\alpha^2 u(\phi)^2+|\nabla_\phi u(\phi)|^2 \right] \, d\sigma.
\end{gather*}
\end{lem}

\begin{prop}\label{p.symmetrized}
Let $\Gamma \subset \partial B_1$ be an open set, then $$\alpha(\Gamma) \ge \alpha(\Gamma^\#),$$
where $\Gamma^\#$ is a spherical cap with $|\Gamma^\#|\le |\Gamma|$. In particular if $\Gamma$ is a spherical cap, then an eigenfunction $w$ corresponding to $\lambda(\Gamma)$ is radially symmetric with respect to the center of $\Gamma$. Namely there exist an isometry $\psi: \partial B_1 \rightarrow \partial B_1$ and a function $w^\mathscr{R}: [0,\pi) \rightarrow \R$ such that 
$$(w\circ \psi)_\mathscr{S}(\theta,\xi)=w^\mathscr{R}(\theta),$$
for every $(\theta,\xi) \in (0,\pi) \times \mathbb{S}^{n-2}$.
\end{prop}
\begin{proof}
Let $u$ be a non-negative eigenfunction corresponding to $\lambda(\Gamma)$ and define the open set
$$\Gamma^+:=\{\phi \in \partial B_1 \, : \, u(\phi)>0\}.$$
We have that $|\Gamma^+| \le |\Gamma|$ and $u$ is a non-negative eigenfunction corresponding to $\lambda(\Gamma^+)$. Consider $u^\#$, the symmetric decreasing rearrangement of $u$, and define the open set
$$\Gamma^\#:=\{\phi \in \partial B_1 \, : \, u^\#(\phi)>0\}.$$
It is a spherical cap, moreover $u$ and $u^\#$ are equidistributed, see \cite[Proposition 1.30 (a) and p. 219]{bdl}, and in particular it holds $|\Gamma^\#|=|\Gamma^+|$. Therefore, by the so-called \emph{P\'olya-Szeg\H{o} inequality} on the sphere, see \cite[Theorem 7.4]{bdl}, we have $u^\# \in H^1_0(\Gamma^\#)$ and from \eqref{spherepseudoeigen} we obtain 
\begin{equation}\label{eq:radialeigen}
    \lambda(\Gamma)
=\lambda(\Gamma^+)
= \ddfrac{\int_{\Gamma^+} |\nabla_\phi u(\phi)|^2 \, d\sigma}{\int_{\Gamma^+} u(\phi)^2 \, d\sigma}
\ge \ddfrac{\int_{\Gamma^\#} |\nabla_\phi u^\#(\phi)|^2 \, d \sigma}{\int_{\Gamma^\#} u^\#(\phi)^2 \, d\sigma}
\ge\lambda(\Gamma^\#),
\end{equation}
by \eqref{def.alpha} we retrieve the first statement. 

For the second statement we notice that if $\Gamma$ is a spherical cap, then there exists an isometry $\psi: \partial B_1 \rightarrow \partial B_1$ such that $\Gamma=\psi(\Gamma^\#)$. Since the inequalities in \eqref{eq:radialeigen} are equalities we have that $u^\#$ and $u \circ\psi$ are eigenfunction corresponding to $\lambda(\Gamma^\#)$. By the connectedness of $\Gamma^\#$ and the fact that $u^\#$ and $u$ are equidistributed we obtain $u^\#=u \circ\psi$, which gives the desired thesis.
\end{proof}
In the following the function $V:(0,\pi) \rightarrow \R$ defined by 
$$V(\theta)= \bigg(\frac{n-2}{2}\bigg)\bigg(\bigg(\frac{n-4}{2}\bigg)(\cot\theta)^2-1\bigg)$$
for every $\theta \in (0,\pi)$, where $\cot(\cdot)$ is the \emph{cotangent function}, will play a central role.
Since the minimum in \eqref{spherepseudoeigen} when $\Gamma$ is a spherical cap is reached by a function that is radially symmetric with respect to the center of $\Gamma$, we can narrow down the set of optimal maps to one-variable (the colatitude) functions.
\begin{lem}\label{l.infspherefunct}
Let $\theta_0 \in (0,\pi)$, then
\begin{equation}\label{eq:newRayleigh}
    \lambda(\theta_0)
    =\min_{\substack{v \in H_0^1((0,\theta_0)) \\ v \not \equiv 0}} \ddfrac{\int_0^{\theta_0} \left[v'(\theta)^2+V(\theta)v(\theta)^2\right]\, d\theta}{\int_0^{\theta_0} v(\theta)^2\, d\theta}.
\end{equation}
There is a unique non-negative function $v$ normalized in $L^2((0,\theta_0))$ achieving the minimum in \eqref{eq:newRayleigh}. This function satisfies
$$v(\theta)=w^\mathscr{R}(\theta)(\sin\theta)^{\frac{n-2}{2}}$$
for every $\theta \in [0,\theta_0]$, where $w$ is an eigenfunction corresponding to $\lambda(\theta_0)$ and $w^\mathscr{R}$ is defined as in Proposition \ref{p.symmetrized}. Moreover it holds
 $v \in C^\infty([0,\theta_0])$, $v(0)=0$, $v(\theta_0)=0$ and $v(\theta)>0$ for every $\theta \in (0,\theta_0)$.
\end{lem}
\begin{proof}
    Let $v \in H_0^1((0,\theta_0))$ such that the objective function in \eqref{eq:newRayleigh} is finite. Recall that by the \emph{Sobolev embedding Theorem} $v \in C([0,\theta_0])$, moreover it is a.e. differentiable in $(0,\theta_0)$ and by \cite[Theorem 2, p. 273]{eva} we have $v(\theta_0)=0$. Let $u: [0,\theta_0] \rightarrow \R$ be the function defined by 
        $$v(\theta)=u(\theta)(\sin\theta)^{\frac{n-2}{2}}$$
    for every $\theta \in [0,\theta_0]$, we have $u(\theta_0)=0$. Differentiating it we obtain
    $$u'(\theta)(\sin\theta)^{\frac{n-2}{2}}=v'(\theta) - \bigg(\frac{n-2}{2}\bigg)\cot(\theta)v(\theta) \quad \text{for a.e. } \theta \in (0,\theta_0).$$
    An integration by part yields
    $$\int_0^{\theta_0} (2v'(\theta)v(\theta)) \cot(\theta) \, d\theta=\int_0^{\theta_0} \frac{v(\theta)^2}{(\sin \theta)^2} d\theta,$$
    therefore by combining these relations we obtain
    $$\ddfrac{\int_0^{\theta_0} u'(\theta)^2(\sin \theta)^{n-2} \, d\theta}{\int_0^{\theta_0} u(\theta)^2(\sin \theta)^{n-2} \, d\theta}=\ddfrac{\int_0^{\theta_0} \left[v'(\theta)^2+V(\theta)v(\theta)^2\right]\, d\theta}{\int_0^{\theta_0} v(\theta)^2\, d\theta}.$$
    Let $w: \overline{\Gamma(\theta_0)} \rightarrow \R$ be the function, radially symmetric with respect to $p$, defined by
    $$w_\mathscr{S}(\theta,\xi)=u(\theta)$$
    for every $(\theta,\xi) \in [0,\theta_0] \times \mathbb{S}^{n-2}$, we have $w=0$ on $\partial \Gamma(\theta_0)$. 
    By \eqref{eq:sphericalchange} and \eqref{Gradient in spherical coordinates} we obtain 
\begin{equation*}
    \ddfrac{\int_{\Gamma(\theta_0)} |\nabla_\phi w(\phi)|^2 \, d\sigma}{\int_{\Gamma(\theta_0)}  w(\phi)^2 \, d\sigma}=\ddfrac{\int_0^{\theta_0} u'(\theta)^2(\sin \theta)^{n-2} \, d\theta}{\int_0^{\theta_0} u(\theta)^2(\sin \theta)^{n-2} \, d\theta},
\end{equation*}
therefore $w \in H_0^1(\Gamma(\theta_0))$. By combining these relations and \eqref{spherepseudoeigen} we have
$$\ddfrac{\int_0^{\theta_0} \left[v'(\theta)^2+V(\theta)v(\theta)^2\right]\, d\theta}{\int_0^{\theta_0} v(\theta)^2\, d\theta}
=\ddfrac{\int_{\Gamma(\theta_0)} |\nabla_\phi w(\phi)|^2 \, d\sigma}{\int_{\Gamma(\theta_0)}  w(\phi)^2 \, d\sigma}\ge \lambda(\theta_0),$$
equality holds if and only if $w$ is an an eigenfunction corresponding to $\lambda(\theta_0)$.

Assume that $w$ is an an eigenfunction corresponding to $\lambda(\theta_0)$, by Proposition \ref{p.symmetrized} and the above relations we have 
\begin{equation}\label{eq.v-w}
    v(\theta)=w^\mathscr{R}(\theta)(\sin\theta)^{\frac{n-2}{2}}
\end{equation}
for every $\theta \in [0,\theta_0]$. Since $\Gamma(\theta_0)$ is a connected set, the eigenvalue $\lambda(\theta_0)$ is simple and $w \ne 0$ in $\Gamma(\theta_0)$. Moreover, since $\Gamma(\theta_0)$ is a set of class $C^\infty$, by \cite[Theorem 1, p. 8]{c} we obtain that $w \in C^\infty(\overline{\Gamma(\theta_0)})$. By imposing the conditions of non-negativity and normalization on $v$, using  \eqref{eq.v-w} and the analyticity of $\mathscr{S}$ we retrieve the desired thesis.
\end{proof}

\section{A convexity property}\label{s.3}
This section is dedicated to the introduction of the main result exploited in the proof of the Friedland-Hayman inequality: the convexity of the one-dimensional function that describes the first eigenvalue of a spherical caps in terms of its colatitude. This fact was proved in a more general setting (and using more sophisticated tools) in \cite[Corollary 1.15]{bl}. We follow the approach of \cite{k}. It is sufficient to retrieve it in dimension greater than or equal to $5$, since it will be exploited in a limiting process in Section \ref{s.4}.

\begin{prop}\label{propconvex}
The function $\lambda: (0,\pi) \rightarrow \R$, defined by \eqref{eq:newRayleigh}
for every $\theta_0 \in (0,\pi)$, is twice differentiable, strictly decreasing and
\begin{equation}\label{lambda(0)}
    \lim_{\theta_0 \rightarrow 0^+} \lambda(\theta_0)=+\infty.
\end{equation}
Moreover if $n \ge 5$, it is a strictly convex function.
\end{prop}
\begin{proof}
Let $\theta_0 \in (0,\pi)$ and denote $v=v_{\theta_0}$ the unique non-negative function normalized in $L^2((0,\theta_0))$ that achieves the minimum in \eqref{eq:newRayleigh}, it holds
\begin{equation*}
    \lambda(\theta_0)=\int_0^{\theta_0} \left[v'(\theta)^2+V(\theta)v(\theta)^2 \right] \, d\theta.
\end{equation*}
By Lemma \ref{l.infspherefunct} we have $v \in C^\infty([0,\theta_0])$, $v(0)=0$, $v(\theta_0)=0$ and $v(\theta)>0$ for every $\theta \in (0,\theta_0)$, which yields
\begin{equation}\label{deriv>0}
   v'(0) > 0 \quad , \quad v'(\theta_0) < 0.
\end{equation}
The function $v$ is the unique non-negative solution of the boundary value problem
\begin{gather}\label{eqbvp0}
\begin{cases}
    -v''(\theta)+V(\theta)v(\theta)=\lambda(\theta_0)v(\theta) &  \text{for every } \theta \in (0,\theta_0), \\
    v(0)=v(\theta_0)=0, \\
    \int_0^{\theta_0} v(\theta)^2 \, d\theta=1. &
\end{cases}
\end{gather}
The first relation in \eqref{eqbvp0} is the so-called \emph{Euler-Lagrange equation} associated to \eqref{eq:newRayleigh}.
Consider the set
$$T=\{ (\theta, \theta_0) \subset [0,\pi] \times (0,\pi) \, : \, \theta \le \theta_0 \},$$
it is possible define a composite function as
$$v: T \rightarrow \R, \quad (\theta,\theta_0) \mapsto v(\theta,\theta_0)=v_{\theta_0}(\theta).$$
We indicate its derivative in the first argument as $v'$, while the one in the second argument as $\dot{v}$. This new function satisfies 
\begin{gather}\label{eqbvp}
\begin{cases}
    -v''(\theta,\theta_0)+V(\theta)v(\theta,\theta_0)=\lambda(\theta_0)v(\theta,\theta_0) &  \text{for every } (\theta,\theta_0) \in \overset{\circ}{T}, \\
    v(0,\theta_0)=v(\theta_0,\theta_0)=0 & \text{for every } \theta_0 \in (0,\pi), \\
    \int_0^{\theta_0} v(\theta,\theta_0)^2 \, d\theta=1 & \text{for every } \theta_0 \in (0,\pi),
\end{cases}
\end{gather}
and it holds
\begin{equation}\label{normeigen}
    \lambda(\theta_0)=\int_0^{\theta_0} \left[ v'(\theta,\theta_0)^2+V(\theta)v(\theta,\theta_0)^2 \right] \, d\theta \quad \text{for every } \theta_0 \in (0,\pi).
\end{equation}
We denote by $\dot{\lambda}$ and $\ddot{\lambda}$ the first and second derivative of the function $\lambda$, respectively. In the following we will always evaluate the function $v$ and its derivatives at a fixed second coordinate $\theta_0$, so, whenever present, the argument of these functions will always correspond to the first one. The proof is divided in five steps.

\medskip
\emph{Step 1 (regularity of $\lambda$).} We claim that $\lambda \in C^\infty(0,\pi)$. By direct calculation we have $\mathscr{X}(\Gamma(\theta_0))= B_{\!\scriptscriptstyle \wedge}$, where $ B_{\!\scriptscriptstyle \wedge}$ is the ball of radius $\tan(\theta_0/2)$ in $\R^{n-1}$ with center in the origin. Let $\tilde u:\R^{n-1} \rightarrow \R$ be a solution of the eigenvalue problem 
\begin{equation}\label{eq:stereoeigenpb0}
    \left\{
\begin{array}{ll}
\dfrac{1}{\Upsilon^2} \left( \Delta \tilde  u + (n-2) \nabla \tilde  u  \cdot \nabla \log \Upsilon \right)+\lambda(\theta_0) \tilde  u=0 & \textup{in }   B_{\!\scriptscriptstyle \wedge},
\\ 
\tilde  u=0 & \textup{on } \partial  B_{\!\scriptscriptstyle \wedge}.
\end{array}
\right.
\end{equation}
By \eqref{Laplacestereo} we notice that the function $\tilde  u_{\mathscr{X}^{-1}}$ is an eigenfunction corresponding to $\lambda(\theta_0)$.
Let $u_{\theta_0}$ be the unique non-negative solution of \eqref{eq:stereoeigenpb0} satisfying the normalization condition
$$\int_{ B_{\!\scriptscriptstyle \wedge}} u_{\theta_0}(\hat x)^2 \,d\hat x=1,$$
we define $\textbf{u}_{\theta_0}:=(u_{\theta_0})_{\mathscr{X}^{-1}}$.
Let $\delta \in \R$ such that $|\delta|\ll\theta_0$, we define the function $h_\delta: \R^{n-1} \rightarrow \R^{n-1}$ as
$$h_\delta(\hat x)=\dfrac{\tan\left(\dfrac{\theta_0+\delta}{2}\right)}{\tan(\theta_0/2)}\hat x$$
for every $\hat x \in \R^{n-1}$.
We notice that $h_\delta(B_{\!\scriptscriptstyle \wedge})=B_{\!\scriptscriptstyle \wedge}^\delta$, where $ B_{\!\scriptscriptstyle \wedge}^\delta$ is the ball of radius $\tan(\theta_0/2)+\delta$ in $\R^{n-1}$ with center in the origin.
Since the operator in \eqref{eq:stereoeigenpb0} has analytic coefficients, is uniformly elliptic on bounded sets, $\lambda(\theta_0)$ is simple and $ B_{\!\scriptscriptstyle \wedge}$ is a set of class $C^\infty$, it is possible to follow the same strategy of \cite[Example 3.2, p. 33]{hen} with the family $\{h_\epsilon\}_{\epsilon \in (-\delta,\delta)}$ to conclude that $\lambda \in C^\infty(0,\pi)$.

\medskip
\emph{Step 2 (regularity of $v$).} We claim that $v \in C^\infty(T)$. Consider the set
$$\mathcal{T}=\{ (\hat x, \theta_0) \subset \R^{n-1} \times (0,\pi) \, : \, |\hat x| \le \tan(\theta_0/2) \},$$
it is possible define a composite function as
$$u: \mathcal{T}  \rightarrow \R, \quad (\hat x,\theta_0) \mapsto u(\hat x,\theta_0)=u_{\theta_0}(\hat x).$$
Let $\underline{u}_{\theta_0}: \R^{n-1} \rightarrow \R$ be the extension of $u_{\theta_0}$ defined in \cite[Example 3.2, p. 33]{hen} exploiting \cite[Theorem 1.9]{hen}, we have $\underline{u}_{\theta_0}(\hat x)=u_{\theta_0}(\hat x)$ for every $\hat x \in \overline{B_{\!\scriptscriptstyle \wedge}}$. The construction of Step 1 allows also to conclude that for every $i \in \N$ the map
\begin{equation}\label{eq:extesionmap}
    (-\delta,\delta) \rightarrow H^i(\R^{n-1}), \quad \epsilon \mapsto \underline{u}_{\theta_0+\epsilon}
\end{equation}
is $C^\infty$. We use the symbol $\,\dot{}\,$ to denote the derivative with respect to the $1$-dimensional parameter used to describe the family $\{h_\epsilon\}_{\epsilon \in (-\delta,\delta)}$. By the \emph{Sobolev embedding Theorem} the functions $\underline{u}_{\theta_0}$ and $\dot{\underline{u}}_{\theta_0}$ belong to $C^\infty( B_{\!\scriptscriptstyle \wedge}^\delta)$. Moreover, from the regularity of the map \eqref{eq:extesionmap} it holds
$$\frac{\underline{u}_{\theta_0+\epsilon}(\hat x)-\underline{u}_{\theta_0}(\hat x)}{\epsilon}=\dot{\underline{u}}_{\theta_0}(\hat x) +\tau_\epsilon(\hat x)$$
for every $(\hat x,\epsilon) \in  B_{\!\scriptscriptstyle \wedge}^\delta \times(-\delta,\delta)$, where the differentiable function $\tau_\epsilon: B_{\!\scriptscriptstyle \wedge}^\delta \rightarrow \R$ satisfies $\tau_\epsilon \rightarrow 0$ in $C^1(B_{\!\scriptscriptstyle \wedge}^\delta)$ as $\epsilon \rightarrow 0$. Differentiating with respect to the coordinates of $\R^{n-1}$ and then letting $\epsilon \rightarrow 0$ we obtain
$$\dot{(\nabla\underline{u}_{\theta_0})}(\hat x)=\nabla\dot{\underline{u}}_{\theta_0}(\hat x)$$
for every $\hat x \in  B_{\!\scriptscriptstyle \wedge}^\delta$. We notice that a similar reasoning can also be applied to higher-order derivatives, therefore restricting ourselves to $\overline{B_{\!\scriptscriptstyle \wedge}}$ we obtain that $u \in C^\infty(\mathcal{T})$.
As a byproduct, since $\mathscr{S}$ and $\mathscr{X}$ are analytic, we have that the function
$$\textbf{u}^\mathscr{R}: T  \rightarrow \R, \quad (\theta,\theta_0) \mapsto \textbf{u}^\mathscr{R}(\theta,\theta_0)=\textbf{u}^\mathscr{R}_{\theta_0}(\theta)$$
belongs to $C^\infty(T)$.
Recall that by Lemma \ref{l.infspherefunct} we have
\begin{equation}\label{eq:vproduct}
    v(\theta,\theta_0)=w^\mathscr{R}_{\theta_0}(\theta)(\sin\theta)^{\frac{n-2}{2}}
\end{equation}
for every $(\theta,\theta_0) \in T$, where $w_{\theta_0}$ is a non-negative eigenfunction corresponding to $\lambda(\theta_0)$. Since both $\textbf{u}_{\theta_0}$ and $w_{\theta_0}$ are non-negative eigenfunctions corresponding to $\lambda(\theta_0)$, then they differs by a positive multiplicative factor. In particular the third condition in \eqref{eqbvp0} yields
$$w^\mathscr{R}_{\theta_0}=\left(\int_0^{\theta_0} \textbf{u}^\mathscr{R}_{\theta_0}(\theta)^2(\sin\theta)^{n-2} \, d\theta \right)^{-1/2}\textbf{u}^\mathscr{R}_{\theta_0},$$
that combined with \eqref{eq:vproduct} gives $v \in C^\infty(T)$.

\medskip
\emph{Step 3 (monotonicity and limit behavior of $\lambda$).}
We claim that the function $\lambda$ is strictly decreasing and that \eqref{lambda(0)} holds. Differentiating the second and the third relation in \eqref{eqbvp} with respect to $\theta_0$ we have
\begin{equation}\label{elementary relations for the convexity theorem}
    \dot{v}(0)=0, \quad v'(\theta_0)=-\dot{v}(\theta_0) \, , \quad \int_0^{\theta
    _0} v(\theta)\dot{v}(\theta) \, d\theta=0.
\end{equation}
Differentiating \eqref{normeigen} we obtain
\begin{align*}
        \dot{\lambda}(\theta_0)&=v'(\theta_0)^2+\int_0^{\theta_0} \left[2 v'(\theta)\dot{v}'(\theta) +2V(\theta)v(\theta)\dot{v}(\theta)\right] \, d\theta,
    \end{align*}
integrating by parts, from the first two relations in \eqref{elementary relations for the convexity theorem}, we have
\begin{equation*}
    \begin{split}
       \dot{\lambda}(\theta_0) &=-v'(\theta_0)^2+2\int_0^{\theta_0} \dot{v}(\theta)[-v''(\theta)+V(\theta)v(\theta)] \, d\theta.
    \end{split}
\end{equation*}
Thence exploiting the first relation in \eqref{eqbvp0} and the third relation in \eqref{elementary relations for the convexity theorem} we find
\begin{equation}\label{signderivmu}
    \dot{\lambda}(\theta_0)=-v'(\theta_0)^2,
\end{equation}
that combined with the second relation in \eqref{deriv>0} gives the desired strict monotonicity for the map $\lambda$. The relation \eqref{lambda(0)} follows from \cite[equation (36), p. 318]{c}.

\medskip
\emph{Step 4 (nodal domains of $\dot{v}$).} We claim that the function $\dot{v}$ has exactly two nodal domains, namely that the connected components of the set $(0,\theta_0)  \setminus  \{\dot{v}= 0\}$ are two.
We introduce the function 
$$q: (0,\theta_0) \rightarrow \R \quad , \quad  q=\frac{\dot{v}}{v},$$
whose derivatives are
\begin{equation*}
    q'=\frac{\dot{v}'}{v}-\frac{v'}{v}q
\end{equation*}
and 
\begin{equation*}
    \begin{split}
        q''&= \frac{\dot{v}''}{v}- \frac{\dot{v}'v'}{v^2}   - \frac{v''}{v}q +\left(\frac{v'}{v}\right)^2q- \frac{v'}{v}q' \\
&=\frac{\dot{v}''}{v}  - \frac{v''}{v}q - 2\frac{v'}{v}q'.
    \end{split}
\end{equation*}
By the first relation in \eqref{eqbvp0} and the derivative of the first relation in \eqref{eqbvp} with respect to the second variable restricted to the set $(0,\theta_0) \times\{\theta_0\}$, namely
\begin{equation*}\label{eqbvp derived with respect to theta0}
    -\dot{v}''(\theta)+V(\theta)\dot{v}(\theta)=\dot{\lambda}(\theta_0)v(\theta)+\lambda(\theta_0)\dot{v}(\theta) \quad \text{for every } \theta \in (0,\theta_0),
\end{equation*}
it is possible to infer
\begin{equation*}
    q''=-\dot{\lambda}(\theta_0)-2\frac{v'}{v}q'
\end{equation*}
and thus
$$(v^2 q')'=-\dot{\lambda}(\theta_0)v^2 .$$
In particular, by Step 3, the function $v^2 q': [0,\theta_0] \rightarrow \R$ is strictly increasing, so it attains its minimum at $0$. Since $v \in C^\infty(T)$, the second relation in \eqref{eqbvp0} and the first relation in \eqref{elementary relations for the convexity theorem} give 
$$\dot{v}'(0),v'(0)<\infty \quad \text{and} \quad \dot{v}(0)=v(0)=0,$$
therefore we have
\begin{equation*}
     (v^2 q')(\theta) 
     > (v^2 q')(0)
     =\dot{v}'(0)v(0) - \dot{v}(0)v'(0)=0 \quad  \text{for every } \theta \in (0,\theta_0).
\end{equation*}
So $q$ is a strictly increasing function and thus it has at most one zero in $(0,\theta_0)$. This latter fact holds also for $\dot{v}$, since $v$ is a positive function in $(0,\theta_0)$. By the third relation in \eqref{elementary relations for the convexity theorem} we notice that $\dot{v}$ is orthogonal to $v$ in $L^2((0,\theta_0))$, therefore $\dot{v}$ has non-constant sign in $(0,\theta_0)$. In particular, since $\dot{v}$ is continuous, there exists a unique value $\bar{\theta} \in (0,\theta_0)$ such that $\dot{v}(\bar{\theta})=0$. At this point we sketch the structure of $\dot{v}$. By the first two relations in \eqref{elementary relations for the convexity theorem} and the second relation in \eqref{deriv>0} we have
$$\dot{v}(0)=0, \quad  \dot{v}\big(\bar{\theta}\big)=0, \quad \dot{v}(\theta_0)=-v'(\theta_0)>0,$$
so
\begin{equation}\label{SignDotv}
    \begin{split}
        \dot{v}(\theta)<0 \quad  \text{for every } \theta \in \big(0,\bar{\theta}\big), \\
    \dot{v}(\theta)>0 \quad \text{for every } \theta \in \big(\bar{\theta},\theta_0\big),
    \end{split}
\end{equation}
and therefore 
\begin{equation}\label{DotvDeriv0}
    \dot{v}'(0)\le 0.
\end{equation}

\medskip
\emph{Step 5 (convexity of $\lambda$).}
We claim that the function $\lambda$ is strictly convex if $n \ge 5$. We notice that it is possible to express the derivative of the map $\lambda$ in another way. Multiplying the first relation in \eqref{eqbvp0} by $v'$ and integrating, by the third relation in \eqref{elementary relations for the convexity theorem} we find 
\begin{equation*}
    -v'(\theta_0)^2 = -\int_0^{\theta_0} V(\theta) \left( v(\theta)^2 \right)' \, d\theta -v'(0)^2,
\end{equation*}
that integrating by parts and using \eqref{signderivmu} yields
\begin{equation}\label{second expression for the derivative of the first eigenvalue in the convexity theorem}
    \dot{\lambda}(\theta_0)= \int_0^{\theta_0} V'(\theta)v(\theta)^2 \, d\theta - v'(0)^2.
\end{equation}
Differentiating \eqref{second expression for the derivative of the first eigenvalue in the convexity theorem}, by the third relation in \eqref{elementary relations for the convexity theorem}, we find
\begin{equation*}
\begin{split}
    \ddot{\lambda}(\theta_0)
    &=2\int_{0}^{\theta_{0}} \left[V'(\theta)-V'\big(\bar{\theta}\big)\right]v(\theta)\dot{v}(\theta) \, d\theta - 2v'(0)\dot{v}'(0) > 0.
\end{split}
\end{equation*}
The last inequality is obtained combining -- recall that $\cot(\cdot)^2$ is a strictly convex function in $(0,\pi)$ -- the strict convexity of $V$ in $(0,\pi)$ for $n \ge 5$, the information contained in \eqref{SignDotv}, the first relation in \eqref{deriv>0} and \eqref{DotvDeriv0}, and gives the desired strict convexity for the map $\lambda$.
\end{proof}

\section{Proof of Theorem \ref{t.ACF}}\label{s.4}
The proof is divided in four steps. For the first two steps we follow the original strategy of \cite[Lemma 5.1]{acf}, while for the remaining ones, which involve the proof of the Friedland-Hayman inequality, we rely on \cite[Section 4.3]{n}.

\medskip
\emph{Step 1 (reduction to an inequality and finiteness).} We claim that the function $J$ is finite and to obtain its monotonicity it is sufficient to prove an inequality.
Let $u \in C^2(B_2)$ be a function satisfying
\begin{equation}\label{propofu}
    \begin{cases}
    \Delta u (x) \ge 0 & \,  \textrm{for every } x \in \{u>0\}, \\
    u(x) \ge 0 & \,  \textrm{for every } x \in B_2, \\
    u(\textbf{0})=0.
    \end{cases}
\end{equation}
Notice that we suppose $u \in C^2(B_2)$, this is done to minimize technical difficulties. However, aside for a passage in this Step and one at the beginning Step 2, that will be appropriately highlighted, the proof can be performed, with the assumption $u \in C(B_2)$. Below, we will repeatedly use \eqref{eq:polarchange} without explicitly stating it. For every $0< s \le 1$, when integrals over subsets of $\partial B_s$ arise, for brevity we write $u$ instead of $u_\mathscr{P}$ and omit the dependence on polar coordinates.
Define the map $I: (0,1) \rightarrow \R$ as
$$I(s)=\int_{B_s} \frac{|\nabla u(x)|^2}{|x|^{n-2}} \,dx \quad $$
for every $s\in (0,1)$, when we want to emphasize the dependence on the function we write $I(\cdot,u)$ instead of $I(\cdot)$. It is possible to prove, through an approximation argument involving mollifiers, that 
$$I(s) \le \frac{\bar C}{s^n} \int _{B_{2s} \setminus B_s} u(x)^2 \, dx$$
for every $s\in (0,1)$, where $\bar C$ is a positive constant depending on $n$, see \cite[equation (12.16)]{cs}. This shows immediately that the value $J(s)$ is finite for every $s \in (0,1)$. Since $u \in C(B_2)$ by the \emph{Caccioppoli inequality} $u \in H^1(B_1)$. So the map $I$ is differentiable almost everywhere, passing to polar coordinates and exploiting \eqref{Gradient in polar coordinates} we have
\begin{equation*}
    I'(s)= s^{2-n} \int_{\partial B_s} \left[ u_r^2+\frac{1}{s^2}|\nabla_\phi u|^2 \right] \, d\sigma 
\end{equation*}
for almost everywhere $s \in (0,1)$. Differentiating \eqref{acffunctional} and evaluating it at one of these points $S \in (0,1)$ for $u^+$ and $u^-$ we obtain
\begin{equation*}
    J'(S)=I(S,u_+)I(S,u_-)S^{-5}\bigg(S \bigg( \frac{I'(S,u_+)}{I(S,u_+)}+\frac{I'(S,u_-)}{I(S,u_-)} \bigg) -4\bigg).
\end{equation*}
Consider the map $u_{\scalebox{0.5}{$\boxed{S}$}}:B_2 \rightarrow \R$ defined as $u_{\scalebox{0.5}{$\boxed{S}$}}(x):=\frac{u(Sx)}{S}$ for every $x \in B_2$, it satisfies \eqref{propofu} with $u=u_{\scalebox{0.5}{$\boxed{S}$}}$ and we can write
\begin{equation}\label{eq:ratioI,I'}
    \frac{I'(S,u)}{I(S,u)}= \frac{1}{S}\ddfrac{\int_{\partial B_1} \left[ ((u_{\scalebox{0.5}{$\boxed{S}$}})_r)^2+|\nabla_\phi (u_{\scalebox{0.5}{$\boxed{S}$}})|^2 \right] \, d\sigma}{\int_{B_1} \frac{|\nabla u_{\scalebox{0.5}{$\boxed{S}$}}(x)|^2}{|x|^{n-2}}  \, dx}.
\end{equation}
In particular to obtain the desired monotonicity it is sufficient to show that
\begin{equation}\label{Step0mono}
     \frac{I'(1,u_+)}{I(1,u_+)}+\frac{I'(1,u_-)}{I(1,u_-)}   -4 \ge 0
\end{equation}
for functions $u_+,u_-$ satisfying \eqref{ACFcond}.

\medskip
\emph{Step 2 (reduction to the Friedland-Hayman inequality).}
We claim that to obtain the monotonicity of the function $J$ it is sufficient to prove the Friedland-Hayman inequality.
Define the open set 
$$\Gamma=\{ u>0\} \cap \partial B_1.$$
Since \eqref{propofu} yields
$$\Delta (u(x)^2) \ge 2|\nabla u(x)|^2 \quad \text{for every } x \in \{u >0\},$$
by the divergence theorem we obtain
\begin{equation*}
\begin{split}
    I(1) \le \frac{1}{2} \int_{B_1} \frac{\Delta u(x)^2}{|x|^{n-2}} \, dx &= \int_{\Gamma}uu_r \, d\sigma + \frac{n-2}{2} \int_{B_1} \frac{\nabla u(x)^2 \cdot x}{|x|^n} \, dx.
\end{split}
\end{equation*}
Moreover, since $\Delta |x|^{2-n}=\bar c \, \delta_\textbf{0}$ in the sense of distribution, where $\bar c$ constant depending on $n$ and $\delta_\textbf{0}$ is the \emph{Dirac delta} evaluated at $\textbf{0}$, and $u(\textbf{0})=0$ we have
$$\int_{B_1} u(x)^2 \Delta \left( \frac{1}{|x|^{n-2}} \right) \, dx =0,$$
hence an integration by part yields
$$\int_{B_1} \frac{\nabla u(x)^2 \cdot x}{|x|^n} \, dx =\int_{\Gamma} u^2 \, d\sigma.$$
In particular we obtain
$$I(1) \le \int_{\Gamma} \left[ uu_r+\frac{n-2}{2} u^2 \right] \, d\sigma,$$
this relation can be retrieved also assuming $u \in C(B_2)$ with a little more work, exploiting an approximation argument involving mollifiers of $u$, see \cite[Lemma 5.1]{acf}.
Consequently by \eqref{eq:ratioI,I'} we have
\begin{equation}\label{quotient}
    \frac{I'(1)}{I(1)} \ge
    \ddfrac{\int_{\Gamma} \left[ u_r^2+|\nabla_\phi u|^2 \right] \, d\sigma}{\int_{\Gamma} \left[ uu_r+\frac{n-2}{2} u^2 \right] \, d\sigma}.
\end{equation}
  Let $t \in [0,1]$ and denote $\lambda=\lambda(\Gamma)$, by \eqref{spherepseudoeigen} and the \emph{Young's inequality} for products we have 
\begin{equation*}
   \int_{\Gamma} \left[ u_r^2+|\nabla_\phi u|^2 \right] \, d\sigma \ge 2 \bigg(\int_{\Gamma} u_r^2 \, d\sigma \bigg)^{\frac{1}{2}}\bigg(t \lambda \int_{\Gamma} u^2 \, d\sigma \bigg)^{\frac{1}{2}} + (1-t) \lambda \int_{\Gamma} u^2 \, d\sigma ,
\end{equation*}
on the other hand by \emph{H\"older's inequality} we obtain
\begin{equation*}
    \int_{\Gamma} uu_r \, d\sigma+\frac{n-2}{2}\int_{\Gamma} u^2 \, d\sigma \le \bigg(\int_{\Gamma} u_r^2 \, d\sigma \bigg)^{\frac{1}{2}}
    \bigg(\int_{\Gamma} u^2 \, d\sigma \bigg)^{\frac{1}{2}}
    +\frac{n-2}{2} \int_{\Gamma} u^2 \, d\sigma.  
\end{equation*}
Therefore setting
$$z=\ddfrac{\bigg(\int_{\Gamma} u^2 \, d\sigma \bigg)^{\frac{1}{2}} }{\bigg(\int_{\Gamma} u_r^2 \, d\sigma \bigg)^{\frac{1}{2}}},$$
it holds
\begin{equation*}
    \frac{I'(1)}{I(1)} \ge \ddfrac{2 (t \lambda)^{\frac{1}{2}}  +  \lambda(1-t) z }{1+\frac{n-2}{2} z},
\end{equation*}
moreover it is possible to estimate
\begin{equation*}
    \ddfrac{2 (t \lambda)^{\frac{1}{2}}  +  \lambda(1-t) z }{1+\frac{n-2}{2} z} \ge 
    2\min \left\{ (t \lambda)^{\frac{1}{2}},\frac{\lambda}{n-2}(1-t) \right\}.
\end{equation*}
At this point, we choose $t$ such that these two lower bounds are equal, namely
\begin{equation}\label{tlambdaCharConstant}
 t\lambda+(n-2)(t \lambda)^{\frac{1}{2}}-\lambda=0,
\end{equation}
this is equivalent to require
\begin{equation*}
    \sqrt{t}=\frac{\sqrt{4\lambda}}{(n-2)+\sqrt{(n-2)^2+4\lambda}}.
\end{equation*}
In particular there exists a unique such $t \in [0,1]$, therefore by \eqref{def.alpha} and \eqref{tlambdaCharConstant} we have $t \lambda=\alpha(\Gamma)^2$ and
\begin{equation}\label{q>alpha}
    \frac{I'(1)}{I(1)} \ge 2\alpha(\Gamma).
\end{equation}
Combining \eqref{Step0mono} and \eqref{q>alpha} we are reduced to show that 
\begin{equation}\label{fhquotient}
     \alpha(\Gamma_+) + \alpha(\Gamma_-) \ge 2,
\end{equation}
for any pair of disjoint open sets $\Gamma_+,\Gamma_- \subset \partial B_1$, i.e., the Friedland-Hayman inequality.

\medskip
\emph{Step 3 (a monotonicity property of characteristic constants).}
We claim that the characteristic constant of a spherical cap of fixed colatitude is monotonically decreasing with respect to its dimension.
Fix $\theta_0 \in (0,\pi)$, we denote $\Gamma_{n}(\theta_0)$ the spherical cap of colatitude $\theta_0$ with center $p$ in $\partial B_1 \subset\R^n$. Let $w: \R^n\rightarrow \R$ be the positive homogeneous function defined by 
$$w_\mathscr{P}(r,\phi)=r^{\alpha(\theta_0,n)}u(\phi),$$
for every $ (r,\phi)\in \R^+ \times \partial B_1 $, where $\alpha(\theta_0,n)$ is the characteristic constant of the set $\Gamma_{n}(\theta_0)$, i.e, $\alpha(\theta_0,n):=\alpha(\Gamma_{n}(\theta_0))$, and $u$ is a positive eigenfunction corresponding to $\lambda(\Gamma_{n}(\theta_0))$. By Proposition \ref{p.harmextens}, the function $w$ is harmonic in $\{w>0\}$, this set is the cone generated by $\Gamma_{n}(\theta_0)$ with vertex in the origin of $\R^n$, in particular $w$ satisfies \eqref{propofu}.
Recall that it is possible to embed the space $\R^n$ into $\R^{n+1}$ using the map
$$(x_1, \ldots, x_n) \rightarrow (x_1, \ldots, x_n,0),$$
in this way the we have $\partial B_1 \subset \mathbb{S}^n$, where $\mathbb{S}^n$ is the sphere of unit radius in $\R^{n+1}$ with center in the origin. We define the function $\tilde w: \R^{n+1} \rightarrow \R$ as
$$\tilde w(x_1, \ldots, x_n,x_{n+1})= w(x_1, \ldots, x_n),$$
it has the same homogeneity degree of $w$ and is harmonic in $\{\tilde w>0\}$. We notice that it holds
$$\{\tilde w>0\}=\{w>0\} \times \R,$$
and so this set is the cone generated by $(\{w>0\} \times \R) \cap \mathbb{S}^n$ with vertex in the origin of $\R^{n+1}$.

\begin{figure}[htp]
    \centering
    
    \begin{minipage}{0.45\textwidth}
        \centering
    \begin{tikzpicture}
\begin{axis}[
    view={120}{30},
    axis lines=none,
    hide axis,
    ticks=none,
    enlargelimits=false,
    clip=false,
]

% Plot the full equator lightly (for reference)
\addplot3[
    domain=0:360,
    samples=30,
    thin,
    color=gray!50,
]
({cos(x)},{sin(x)},{0});

% Blue arc (0° to 300°)
\addplot3[
    domain=0:300,
    samples=30,
    thick,
    color=blue,
    line join=round,
    mesh/ordering=y varies,
]
({cos(x)},{sin(x)},{0});

% Orange arc (300° to 360°)
\addplot3[
    domain=300:360,
    samples=30,
    thick,
    color=orange,
    line join=round,
    mesh/ordering=y varies,
]
({cos(x)},{sin(x)},{0});

% Optional: dashed lines from origin
\draw[dashed, orange!70!black] (axis cs:0,0,0) -- (axis cs:{cos(300)},{sin(300)},0);
\draw[dashed, orange!70!black] (axis cs:0,0,0) -- (axis cs:{cos(360)},{sin(360)},0);

\end{axis}
\end{tikzpicture}   
    \end{minipage}
    \hfill
    \begin{minipage}{0.45\textwidth}
        \centering
    \begin{tikzpicture}

\begin{axis}[
    view={120}{30},
    hide axis,
    axis lines=none,
    colormap/blackwhite,
    enlargelimits=false,
    clip=false,
    z buffer=sort,
    samples=30,
    samples y=35,
    domain y=0:180,
]

% Complement of the orange slice (main sphere part)
\addplot3[
    surf,
    opacity=0.4,
    color=blue,
    domain=0:300, % <-- azimuthal domain for the remaining part
]
(
    {sin(y)*cos(x)},
    {sin(y)*sin(x)},
    {cos(y)}
);

% Orange slice (the removed wedge)
\addplot3[
    surf,
    opacity=0.9,
    color=orange,
    domain=300:360, % <-- this defines the slice
]
(
    {sin(y)*cos(x)},
    {sin(y)*sin(x)},
    {cos(y)}
);

\end{axis}
\end{tikzpicture}
    \end{minipage}

    \caption{The set $\partial B_1 \cap \{w>0\}$ in the case $n=2$ and $\theta_0 \approx \pi/3$ (in orange, on the left). 
    The set $\mathbb{S}^n \cap \{\tilde w>0\}$ in the case $n=2$ and $\theta_0 \approx\pi/3$ (in orange, on the right).}
    \label{fig:placeholder}
\end{figure}

Moreover we can write
$$\tilde w_\mathscr{\tilde S}(\tilde r,\tilde \phi)=\tilde{r}^{\alpha(\theta_0,n)}\tilde u(\tilde\phi),$$
for every $ (\tilde r,\tilde \phi)\in \R^+ \times \mathbb{S}^n $, where $\mathscr{\tilde S}$ is the polar parametrization of $\R^{n+1}$ (and $(\tilde r,\tilde \phi)$ the corresponding polar coordinates) and $\tilde u: \mathbb{S}^n \rightarrow \R$ is a function. We observe that $\tilde w$ satisfies the analogous of  \eqref{propofu} in $\R^{n+1}$, so by Step 2, see \eqref{q>alpha}, we obtain
$$\frac{I'(1,\tilde w)}{I(1,\tilde w)} \ge 2 \alpha(\theta_0,n+1),$$
on the other hand
by \eqref{eq:ratioI,I'} and Lemma \ref{l.inthom} we have
$$\frac{I'(1,\tilde w)}{I(1,\tilde w)}=\frac{I'(1,w)}{I(1,w)}=2\alpha(\theta_0,n),$$
which yields
\begin{equation}\label{monoalpha}
    \alpha(\theta_0,n) \ge \alpha(\theta_0,n+1).
\end{equation}

\medskip
\emph{Step 4 (the Friedland-Hayman inequality).}
We claim that the Friedland-Hayman inequality holds. Consider \eqref{fhquotient}, by Proposition \ref{p.symmetrized} it is sufficient to prove that it holds for all pairs $(\Gamma_+,\Gamma_-)$ of disjoint spherical caps in $\partial B_1$. From the strict monotonicity statement of Proposition \ref{propconvex}, (an isometry) and \eqref{def.alpha} we can take these sets to be complementary in $\partial B_1$, namely this is equivalent to prove that
\begin{equation*}
    \min_{\theta_0 \in (0,\pi)} \alpha(\theta_0,n)+\alpha(\pi-\theta_0,n) \ge 2.
\end{equation*}
By \eqref{lambda(0)} and the differentiability statement of Proposition \ref{propconvex} there exists a value $\theta_n \in (0,\pi)$ such that the minimum is achieved. Define the function $\beta : \{ m \in \N \, : \, m \ge 3 \} \rightarrow \R$ as
$$\beta(n)=\alpha(\theta_n,n)+\alpha(\pi-\theta_n,n),$$
for $n \in \{ m \in \N \, : \, m \ge 3 \}$, by \eqref{monoalpha} it is monotone non-increasing.

Suppose by way of contradiction that exists an $n_0 \in \N$ and $\delta >0$ such that $\beta(n_0)<2-\delta$, then $\beta(n)<2-\delta$ for all $n \ge  n_0$ (this value will be increased in the following in order to satisfy more conditions and ease the notation). Therefore by minimality of $\theta_n$ we have
$$\alpha(\theta_n,n),\alpha(\pi-\theta_n,n)<2,$$
that by \eqref{def.alpha} gives
\begin{equation}\label{est.lambda_n}
    \lambda(\theta_n),\lambda(\pi-\theta_n)<2n.
\end{equation}
We study the behavior of the sequence
$$\left\{\beta(n)=\bigg( \dfrac{n-2}{2} \bigg)\left(\sqrt{1 + \frac{4\lambda(\theta_n)}{(n-2)^2}}- 1 + \sqrt{1 + \frac{4\lambda(\pi-\theta_n)}{(n-2)^2}}- 1\right) \right\}_n$$
as $n \rightarrow +\infty$, we can estimate
\begin{equation*}
    \beta(n) \ge \frac{\lambda(\theta_n)+\lambda(\pi-\theta_n)}{n-2}-\frac{\lambda(\theta_n)^2+\lambda(\pi-\theta_n)^2}{(n-2)^3}
\end{equation*}
for all $n \ge n_0$, up to taking a bigger $n_0$, since the third term in the Taylor formula -- the first one not appearing here -- is positive. By \eqref{est.lambda_n} we can estimate
\begin{equation*}
    -\frac{\lambda(\theta_n)^2+\lambda(\pi-\theta_n)^2}{(n-2)^3} \ge -\frac{c}{n},
\end{equation*}
where $c$ is a positive constant (not depending on $n$). Consider the function $\gamma_n: (0,\pi) \rightarrow \R$ defined as 
$$\gamma_n(\theta):=\frac{\lambda(\theta)+\lambda(\pi-\theta)}{n-2},$$
for every $\theta \in (0,\pi)$. By Proposition \ref{propconvex}, for $n \ge 5$, $\gamma_n$ is a strictly convex function, evenly symmetric with respect to the point $\pi/2$, that is its unique minimum. So by \eqref{eigenhalfcap} we obtain
\begin{equation*}
    \gamma_n(\theta_n) \ge \gamma_n\left(\frac{\pi}{2}\right)= \frac{2n-2}{n-2} > 2
\end{equation*}
for all $n\ge 5$, therefore by combining the estimates above we have
$$\beta(n)>2 - \frac{c}{n} \ge 2-\delta,$$
for all $n \ge n_0$, up to taking a bigger $n_0$, a contradiction. This proves the Friedland-Hayman inequality and therefore, by Step 2, the desired monotonicity of $J$. \hfill \qed

\section{An open problem}\label{s.5}
We present an open question taken from \cite[Introduction]{ak}.

We notice that Theorems \ref{t.rigidity} and \ref{t.stability} describe the structure of the functions $u_+,u_-$ satisfying \eqref{ACFcond} when the map $J$ is constant and when the map $J$ is near to a constant, respectively. In particular a family of functions, the one constituted by the so-called \emph{two-planes solutions}, is central in the description of the behavior of the pair $u_+,u_-$. This family is the set
$$
\left\{
\begin{array}{ll}
\begin{aligned}
\ell : \mathbb{R}^n &\rightarrow \mathbb{R} \\
x &\mapsto  c_+(x \cdot \nu)^+ + c_-(x \cdot \nu)^- 
\end{aligned}
&
\Bigg|\quad \raisebox{0\height}{$c_+,c_->0 \textup{ and } \nu \in \partial B_1$}
\end{array}
\right\}.
$$
A two-planes solution is a function made up of a pair of positive linear functions defined on two disjoint complementary half-spaces, vanishing on the common boundary that contains the origin, they satisfy \eqref{ACFcond}.
\begin{figure}[htp]
        \centering

       % Improved 3D viewing angle: more top-down to show V shape clearly
\tdplotsetmaincoords{65}{100}

\begin{tikzpicture}[tdplot_main_coords, scale=2.5]

% Axes
\draw[->] (-1.5,0,0) -- (1.5,0,0) node[below right] {$x_1$};  % Spine
\draw[->] (0,-1.2,0) -- (0,1.2,0) node[below right] {$x_2$};

% Parameters
\def\xmax{1.2}
\def\ymax{1.0}
\def\cp{0.8}  % c_+ slope
\def\cm{-0.4}  % c_- slope

% Upper half (x2 > 0): u_+(x) = c_+ x_2
\filldraw[fill=blue!40, opacity=0.7]
  (-\xmax,0,0)
  -- (\xmax,0,0)
  -- (\xmax,\ymax,\cp*\ymax)
  -- (-\xmax,\ymax,\cp*\ymax)
  -- cycle;

% Lower half (x2 < 0): u_-(x) = c_- x_2
\filldraw[fill=red!40, opacity=0.7]
  (-\xmax,0,0)
  -- (\xmax,0,0)
  -- (\xmax,-\ymax,-\cm*\ymax)
  -- (-\xmax,-\ymax,-\cm*\ymax)
  -- cycle;

% Origin
\fill (0,0,0) circle (0.5pt);
\node[below right] at (0,-0.05,0) {$\mathbf{0}$}; % moved slightly into red side

% Function labels (adjusted positions)
%\node[blue!70!black] at (0.4, 1.7, \cp*\ymax + 0.2) {$u_+(x) = c_+ x_2^+$};
%\node[red!70!black] at (-0.5, -1.7, -\cm*\ymax - 0.25) {$u_-(x) = c_- x_2^-$};

% Optional: highlight the V-crease (spine)
\draw[dashed, thick, gray!60] (-\xmax,0,0) -- (\xmax,0,0);

\end{tikzpicture}
        \caption{An example of two-planes solution in $\R^2$.}
        \label{Two plane solution in the two dimensional case}
    \end{figure}

Let $u_+,u_-$ be two functions satisfying \eqref{ACFcond} and 
\begin{equation}\label{posACFlim}
    \lim_{s \rightarrow 0^+} J(s)>0.
\end{equation}
Consider $\{s_k\}_k \subset (0,1)$ a \emph{sequence of radii} decreasing to $0$, and the associated \emph{blow-up sequences}
$$\left\{\frac{u_+(s_k \,\cdot)}{s_k}\right\}_k \quad \text{and} \quad \left\{\frac{u_-(s_k \,\cdot)}{s_k}\right\}_k,$$
restricted to $B_1$. As in \cite[Section 6]{acf}, up to subsequence, it is possible to show that they converge, in an appropriate way, to the the functions 
$$x \mapsto c_+(x \cdot \nu)^+ \quad \text{and} \quad x \mapsto c_-(x \cdot \nu)^-,$$
where $c_+,c_->0$ and $\nu \in \partial B_1$, defined for $x \in B_1$. This pair is called \emph{blow-up limit}. Observe that the values $c_+,c_-,\nu$, depend \emph{a priori} on the sequence $\{s_k\}_k$. The result of the blow-up procedure described above is a pair of functions representing one of the possible behaviors that the original pair exhibits when approaching the origin. In the theory of free boundary problems, classifying the possible blow-up limits is useful to obtain information about the regularity of the free boundary, see for example \cite[proof of Theorem 1.1]{dpsv} for a related result.

It has been shown in \cite[Theorem 1.3]{ak} that it is possible to construct two functions $\tilde{u}_+$ and $\tilde{u}_-$ satisfying \eqref{ACFcond}, that admit multiple different blow-up limits, depending on the chosen sequence of radii. In particular the positivity sets of these functions wrap around the origin. More precisely for all of the possible blow-up limits, $\tilde{c}_+$ and $\tilde{c}_-$ are the same, while $\tilde{\nu}$ can be different, depending on the particular subsequence of radii.

\medskip
\begin{op}[\cite{ak}]
    Is it possible to find a pair $u_+,u_-$ satisfying \eqref{ACFcond} and \eqref{posACFlim} such that for all of the blow-up limits, $\nu$ is the same, while the pair $(c_+,c_-)$ can be different, depending on the particular subsequence of radii?
\end{op}

\medskip

\noindent{\textbf{Use of Generative-AI tools declaration.}}
The author declares the use of Ar
tificial Intelligence (AI) tools in the creation of this article. The author has used
ChatGPT (Free Plan, model GPT-4o) to generate parts of the TikZ code that produces the figures in the article.

\medskip

\noindent{\textbf{Acknowledgements.}}
The author would like to express his sincere gratitude to an anonymous reviewer who carefully read this article and suggested several important improvements (including the reference \cite{v}). The author wishes to thank Professor Mark Allen for providing him with feedback on the open problem of \cite{ak}. The author wishes to thank Professor Benedetta Noris for providing him with feedback on \cite{n}. The author wishes to thank Professor Susanna Terracini for proposing this topic as master thesis subject. The author wishes to thank Professor Giorgio Tortone for the useful conversations about this topic. This publication is part of the project PNRR-NGEU which has received funding from the MUR – DM 351/2022. The author is a member of the Research Group INdAM–GNAMPA. The author has been partially supported by the GNAMPA 2024 project \emph{Free boundary problems in non-commutative structures and degenerate operators}.


\begin{thebibliography}{11}

\small

\bibitem{ak} M.~Allen and D.~Kriventsov. A spiral interface with positive Alt-Caffarelli-Friedman limit at the origin. Analysis and PDE 13 (2018) 201--214.

\bibitem{akn} M.~Allen, D.~Kriventsov, and R.~Neumayer. Sharp quantitative Faber-Krahn inequalities and the Alt-Caffarelli-Friedman monotonicity formula. Ars Inveniendi Analytica 49 (2023).

\bibitem{acf1} H.~W.~Alt, L.~A.~Caffarelli, and A.~Friedman. Jets with two fluids {I}: {O}ne free boundary. Indiana University Mathematics Journal 33 (1984) 213--247.

\bibitem{acf} H.~W.~Alt, L.~A.~Caffarelli, and A.~Friedman. Variational problems with two phases and their free boundaries. Transactions of the American
Mathematical Society 282 (1984) 431--461.

\bibitem{bdl} A.~Baernstein II, D.~Drasin, and R.~Laugesen. Symmetrization in analysis. New Mathematical Monographs 36, Cambridge University Press, Cambridge, 2019.

\bibitem{bkp} W.~Beckner, C.~E.~Kenig and J.~Pipher. A convexity property for Gaussian measures. 1998.

\bibitem{bjs} L.~Bers, F.~John, and M.~Schechter. Partial differential equations. American Mathematical Society, Providence, RI, 1979.

\bibitem{bl} H.~Brascamp and E.~Lieb. Some inequalities for Gaussian measures and the long-range order of the one-dimensional plasma, 403--416. In M.~Loss and M.~B.~Ruskai, Inequalities: Selecta of Elliott H. Lieb. Springer, 2002.

\bibitem{ck} L.~A.~Caffarelli and C.~E.~Kenig. Gradient estimates for variable coefficient parabolic equations and singular perturbation problems. American Journal of Mathematics 120 (1998) 391--439. 

\bibitem{cs} L.~A.~Caffarelli and S.~Salsa. A geometric approach to free boundary problems. Graduate Studies in Mathematics 68, American Mathematical Society, Providence, RI, 2005.

\bibitem{c} I.~Chavel. Eigenvalues in Riemannian geometry. Pure and Applied Mathematics, Academic Press, Inc., Orlando, FL, 1984.

\bibitem{dpsv} G.~De Philippis, L.~Spolaor, and B.~Velichkov. Regularity of the free boundary for the two-phase Bernoulli problem. Inventiones Mathematicae 225 (2021) 347--394.

\bibitem{em} Á.~Elbert and M.~E.~Muldoon. Inequalities and monotonicity properties for zeros of Hermite functions. Proceedings of the Royal Society of
Edinburgh Section A 129 (1999) 57--75.

\bibitem{eva} L.~C.~Evans. Partial differential equations. Graduate Studies in Mathematics, American Mathematical Society, Providence, RI, 1998.

\bibitem{fy} D.~Feng and X.~Yuan. Approximation theory and harmonic analysis on spheres and balls. Springer Monographs in Mathematics, Springer, New York, 2013.

\bibitem{fol} G.~B.~Folland. Real analysis, second edition. Pure and Applied Mathematics, A Wiley-Interscience Series of Texts, Monographs, and Tracts, John Wiley \& Sons, Inc., New York, 1999.

\bibitem{fh} S.~Friedland and W.~K.~Hayman. Eigenvalue inequalities for the Dirichlet problem on spheres and the growth of subharmonic functions.
Commentarii Mathematici Helvetici 51 (1976) 133--161.

\bibitem{ho} W.~K.~Hayman and E.~L.~Ortiz. An upper bound for the largest zero
of Hermite’s function with applications to subharmonic functions. Proceedings of the Royal Society of Edinburgh Section A 75 (1976) 183--197.

\bibitem{hen} D.~Henry. Perturbation of the boundary in boundary-value problems of partial differential equations. London Mathematical Society Lecture Note Series, Cambridge University Press, Cambridge, 2005.

\bibitem{k} H.~Koch. Convexity and concavity of the ground state energy. New York Journal of Mathematics 21 (2015) 1003--1005.

\bibitem{lee} J.~M.~Lee. Riemannian manifolds. Graduate Texts in Mathematics 176, Springer-Verlag, New York, 1997.

\bibitem{n} B.~Noris. Il lemma di monotonia. Master thesis. University of Milano-Bicocca, 2005.

\bibitem{pet} P.~Petersen. Riemannian geometry. Graduate Texts in Mathematics 171, Third edition, Springer, Cham, 2016.

\bibitem{psu} A.~Petrosyan, H.~Shahgholian, and N.~Uraltseva. Regularity of free
boundaries in obstacle-type problems. Graduate Studies in Mathematics 136, American Mathematical Society, Providence, RI, 2012.

\bibitem{v} B.~Velichkov. A note on the monotonicity formula of
{C}affarelli-{J}erison-{K}enig. Atti della Accademia Nazionale dei Lincei. Rendiconti Lincei. Matematica e Applicazioni 25 (2014) 165--189.

\end{thebibliography}
\end{document}